\documentclass[a4paper]{scrartcl}

\usepackage[T1]{fontenc}

\usepackage{amsthm}
\usepackage{amsmath}
\usepackage{amsfonts}
\usepackage{mathtools}
\usepackage{amssymb}

\usepackage{caption}

\usepackage{enumerate}

\usepackage{amscd} 

\usepackage{tikz}
\usetikzlibrary{calc}
\usetikzlibrary{cd} 

\usepackage[all]{xy}
\SelectTips{cm}{}

\usepackage{graphicx}
\usepackage[]{xcolor}

\usepackage[pagebackref,
    ,pdfborder={0 0 0}
  ,urlcolor=black,hypertexnames=false,
pdftitle={Torsion homology growth and cheap rebuilding of inner-amenable groups}]{hyperref}

\usepackage[]{scrlayer-scrpage}

\automark[section]{section}
\clearpairofpagestyles
\ihead{\headmark}
\ohead{\pagemark}

\newtheoremstyle{ggt}{}{}{\itshape}{}{\sffamily\bfseries}{.}{ }{}
\newtheoremstyle{ggtdefinition}{}{}{}{}{\sffamily\bfseries}{.}{ }{}

\theoremstyle{ggt}
\newtheorem{thm}{Theorem}[section]
\newtheorem{lemma}[thm]{Lemma}
\newtheorem{cor}[thm]{Corollary}
\newtheorem{prop}[thm]{Proposition}

\theoremstyle{ggtdefinition}
\newtheorem{defi}[thm]{Definition}
\newtheorem{example}[thm]{Example}
\newtheorem{question}[thm]{Question}

\newtheorem{remark}[thm]{Remark}
\newtheorem*{ack}{Acknowledgements}

\def\N{\mathbb{N}}
\def\Z{\mathbb{Z}}
\def\Q{\mathbb{Q}}

\def\actson{\curvearrowright}

\newcommand{\field}{\ensuremath{\mathbb{K}}}

\newcommand{\careb}{cheap $\alpha$-rebuilding property}
\newcommand{\coreb}{cheap 1-rebuilding property}
\newcommand{\ctreb}{cheap 2-rebuilding property}

\newcommand{\namen}{non-amenable}

\newcommand{\ABFG}{Abért, Bergeron, Fr\k{a}czyk and Gaboriau}

\def\qand{\quad\text{and}\quad}



\makeatletter
\def\blfootnote{\xdef\@thefnmark{}\@footnotetext}
\makeatother
\def\draftinfo{}
\date{\today%
  \protect\blfootnote{\copyright{\ M.~Uschold~2022}. 
    This work was supported by the CRC~1085 \emph{Higher Invariants} 
    (Universit\"at Regensburg, funded by the DFG). 
    \\
    MSC~2020 classification: 57M07, 20E26, 43A07, 
    \draftinfo}}

\allowdisplaybreaks[1]
  
\begin{document}

\title{Torsion homology growth and cheap rebuilding of~inner-amenable groups}
\author{Matthias Uschold}

\maketitle 

\thispagestyle{empty}

\begin{abstract}
	We prove that virtually torsion-free, residually finite groups that are inner-amenable and non-amenable 
	 have  the \coreb, a notion recently introduced by \ABFG. As a consequence, the first $\ell^2$-Betti number with arbitrary 
	 field coefficients
	and log-torsion in degree $1$ vanish for these groups. This extends results previously 
	known for amenable groups to inner-amenable groups.	
	We use a structure theorem of Tucker-Drob for inner-amenable groups
	showing the existence of a chain of $q$-normal subgroups.
\end{abstract}

\section{Introduction}

In 1994, Lück proved the approximation theorem about $\ell^2$-Betti numbers. Its 
group-theoretic version asserts that for a residually finite group $\Gamma$ with finite type model
for $E\Gamma$, and any residual chain $(\Gamma_i)_{i\in \N}$ of $\Gamma$ (i.e., a chain of nested, normal, finite 
index subgroups in $G$ whose intersection is trivial), the following holds: for all $n\in \N$,
we have 
\cite[Theorem~0.1]{Lueck-Approx-Thm}\ \cite[Theorem~5.3]{Kammeyer-L2}
\[
	b_n^{(2)}(\Gamma) = \lim_{i\to\infty} \frac{b_n(\Gamma_i)}{[\Gamma:\Gamma_i]}.
\] 
Here,  $b_n(\cdot)$ denotes the (ordinary) $n$-th Betti number and $b_n^{(2)}(\cdot)$ the $\ell^2$-Betti number as defined by Atiyah \cite{Atiyah-L2} originally for spaces (for an introduction, see the book by Kammeyer \cite{Kammeyer-L2}). 

We obtain a different viewpoint by taking the right hand side of this equality as a definition.
The main advantage is that we can replace the $n$-th ordinary Betti number by different 
homology-related invariants, e.g. we can consider the invariants 
\[
	\limsup_{i\to\infty} \frac{b_n(\Gamma_i,\field)}{[\Gamma:\Gamma_i]} \quad \text{and}
	\quad
	\limsup_{i\to\infty} \frac{\log |H_n(\Gamma_i, \Z)_\mathrm{tors}|}{[\Gamma:\Gamma_i]},
\]
where $\field$ is any field and $\operatorname{tors}$ 
denotes the torsion subgroup of an abelian group. 
We call the resulting invariant the \emph{gradient} of the invariant
that we insert instead of ordinary Betti numbers. 
For these gradient invariants, a priori, we do not know whether the $\limsup$ is actually 
a proper limit or if its value depends on the chosen residual chain.
The Betti number gradients do however depend on the field: 
Avramidi, Okun and Schreve exhibited an example of a right-angled Artin group
where the $\Q$-Betti number gradients (i.e.\ the $\ell^2$-Betti numbers) and the $\mathbb{F}_2$-Betti
number gradients do \emph{not} coincide \cite[Corollary~2]{AOS-tors-hom-growth}.

An efficient way to show vanishing of Betti number gradients for all fields and vanishing 
of log-torsion gradients is via the \careb, recently introduced by \ABFG\ \cite{ABFG}. Roughly, for a fixed $\alpha\in \N$, a group~$\Gamma$ has the \careb\ if 
for all Farber sequence $(\Gamma_i)_{i\in \N}$, the following holds in a uniform way: Because the $\Gamma_i$ 
are finite index subgroups, we obtain a tower of finite degree coverings $B\Gamma_i \to B\Gamma$. 
If $i$ is large enough, we can find a model of $B\Gamma_i$ (i.e.\ a CW-complex that of the
homotopy type of $B\Gamma_i$) with few cells up to dimension~$\alpha$, maintaining tame norms
on the cellular boundary operators and homotopies. 
(We will give a more precise definition of the \careb\ in Appendix~\ref{sec:careb}.)

In this article, we will prove that certain inner-amenable groups have the \coreb.
Inner-amenable groups are defined as follows and generalise the notion of amenability of groups.

\begin{defi}[inner amenability, {{\cite[Definition~0.7]{Tucker-Drob2020}}}]
	\label{def:inner-amen}
	A group $\Gamma$ is \emph{inner-amenable} if the conjugation action of $\Gamma$ on itself 
	admits an atomless invariant mean, i.e., if there is a finitely additive probability measure
	$\mu\colon {\cal P}(\Gamma) \to [0,1]$ such that for all subsets $A\subset \Gamma$ and
	 $g\in \Gamma$, we have 
	\[
		\mu(g\cdot A\cdot g^{-1}) = \mu (A)
	\] and additionally $\mu(\{x\}) = 0$ for all $x\in \Gamma$.
\end{defi}

We collect examples of such groups in Section~\ref{sec:inner-amenable}. Our results
follow from a structure result of Tucker-Drob \cite{Tucker-Drob2020} for this class of groups
(Theorem~\ref{thm:structure-inner-am}) about the existence of $q$-normal subgroups, 
suggesting a strategy for proving properties for 
inner-amenable groups (Theorem~\ref{thm:Tucker-Drob-method}).

\subsection*{Main results}

Recall that a group is \emph{virtually torsion-free} if there exists a torsion-free
subgroup of finite index.

\begin{thm}[Theorem~\ref{thm:in-am-coreb}] 
	\label{thm:main:in-am-coreb}
	Let $\Gamma$ be a finitely generated, virtually torsion-free, residually finite group
	that is inner-amenable and \namen. 
	Then, $\Gamma$ has the \coreb.
\end{thm}

In particular, we have the following corollary.

\begin{cor}
	\label{cor:main:vanishing}
	Let $\Gamma$ be a finitely \emph{presented}, virtually torsion-free, residually finite, inner-amenable group. 
	Then, for every Farber sequence $(\Gamma_i)_{i\in \N}$ and every coefficient field $\field$,
	 we have
	\[
		\lim_{i\to\infty} \frac{\dim_K H_1(\Gamma_i, \field)}{[\Gamma:\Gamma_i]} = 0 \qand
		\lim_{i\to\infty} \frac{\log |H_1(\Gamma_i, \Z )_\mathrm{tors}|}{[\Gamma:\Gamma_i]} = 0.
	\]
\end{cor}

\begin{proof}
	If $\Gamma$ is amenable, the claims are already known: The second claim
	about torsion growth was proved by Kar, Kropholler and Nikolov \cite[Theorem~1]{Kar-Kropholler-Nikolov}. To derive the 
	first claim, we can apply \cite[Remark~1.3]{Lueck-Approx-L2-2013} and we obtain that it suffices to prove
	the claim for $K=\Q$. By Lück's approximation theorem \cite{Lueck-Approx-Thm}, 
	it suffices to show
	that $\ell^2$-Betti numbers of amenable groups vanish (in degree $1$) \cite[Theorem~0.2]{Cheeger-Gromov}.
	
	If $\Gamma$ is \namen, 
	by Theorem~\ref{thm:main:in-am-coreb}, $\Gamma$ has the cheap $1$-rebuilding property. 
	Because $\Gamma$ is finitely
	presented, it is of type $F_2$. Thus, the corollary follows from the 
	work of \ABFG\ \cite[Theorem~10.20]{ABFG}.
\end{proof}

Note that we prove a stronger assertion for inner-amenable groups that are \namen. 
It is unknown if all amenable groups have the \coreb\ \cite[Question~10.21]{ABFG}.

By work of Chifan, Sinclair and Udrea, using
ergodic-theoretic methods, it was already known that
for all countable inner-amenable groups, the first $\ell^2$-Betti number
vanishes \cite[Corollary~D]{Chifan-Sinclair-Udrea2016}. Later, Tucker-Drob proved that 
the \emph{cost} of these
groups is equal to $1$ \cite[Theorem~5]{Tucker-Drob2020}, thus implying the same result.
In this article, we obtain a more restrictive result, as we require the groups in
question to be virtually torsion-free and residually finite. 
However, we can present a topological proof
without using ergodic theoretic methods.

We point out that the proof of the amenable case in
 Corollary~\ref{cor:main:vanishing} only uses results that hold in all degrees (instead of just 
 degree~$1$). It is thus natural to ask the following question.

\begin{question}
	\label{q:higher-degrees-careb}
	Let $\alpha\in \N$ and let $\Gamma$ be a virtually torsion-free, residually finite, inner-amenable group of type $F_\alpha$. Does $\Gamma$ have the \careb? 
\end{question}

If the answer should be negative, at least the analogue of Corollary~\ref{cor:main:vanishing}
could be true.

\begin{question}
	\label{q:higher-degrees-vanish}
	If additionally, $\Gamma$ is of type $F_{\alpha+1}$, does the following hold? 
	For every Farber sequence $(\Gamma_i)_{i\in \N}$ and every coefficient field $\field$, we have
	\[
		\lim_{i\to\infty} \frac{\dim_K H_n(\Gamma_i, \field)}{[\Gamma:\Gamma_i]} = 0 \quad \text{and} \quad 
		\lim_{i\to\infty} \frac{\log |H_n(\Gamma_i, \Z )_\mathrm{tors}|}{[\Gamma:\Gamma_i]} = 0.
	\]
\end{question}

For right-angled Artin groups, we have a positive answer.

\begin{prop}[Corollary~\ref{cor:raag}] 
	\label{prop:main:raag}
	Let $\Gamma$ be a right-angled Artin group.
	If $\Gamma$ is inner-amenable, then $\Gamma$ has the \careb\ for all $\alpha\in \N$.
	In particular, the conclusions of Question~\ref{q:higher-degrees-vanish} hold for $\Gamma$.
\end{prop}

\subsection*{Organisation of this article}

In Section~\ref{sec:inner-amenable}, we collect many examples for inner-amenable groups.

Theorem~\ref{thm:main:in-am-coreb} is proved in Section~\ref{sec:app-coreb},
using a strategy by transport through $q$-normality outlined in Section~\ref{sec:tucker-drob}. 
For this, we 
need to construct actions out of $q$-normality (Section~\ref{sec:act-q-normal}). We explain
how to obtain such an action using a more general principle in Section~\ref{sec:more-actions}.

We will prove Proposition~\ref{prop:main:raag} in Section~\ref{sec:outlook}, where we also give 
an outlook on why these results might not generalise to higher degrees.

In Appendix~\ref{sec:careb}, we recall the definition of the \careb.

\begin{ack}
	The author is very grateful to his advisor, Clara Löh, for suggesting the generalisation of 
	properties to inner-amenable groups,  encouraging an axiomatisation of Tucker-Drob's
	arguments as well as helpful and encouraging discussions.
	He would like to express his gratitude to Kevin Li for simplifying some steps in the argument,
	suggesting the shortcut in Section~\ref{sec:act-q-normal} 
	as well as many discussions about cheap rebuilding of classifying spaces.
	The author is grateful to Clara Löh, Kevin Li and 
	Damien Gaboriau for comments on a preliminary version of 
	this article. 
	The author would also like to thank José Pedro Quintanilha for a discussion about actions on universal coverings.
	Thanks also to Zhicheng Han, who suggested Example~\ref{ex:han}, to Francesco Fournier-Facio 
	for pointing to Example~\ref{ex:fff2} and to the anonymous referee 
	for helpful suggestions.
\end{ack}

\section{Inner-amenable groups}
\label{sec:inner-amenable}

Inner-amenable groups were originally defined by Effros \cite{Effros}, who showed that property 
Gamma implies inner amenability. Effros' question whether the converse holds was answered 
negatively by Vaes \cite{Vaes-inneram-not-Gamma}.

We define inner-amenable groups by the existence of an \emph{atomless, conjugation-invariant} mean (see
Definition~\ref{def:inner-amen}). Note that this is the special case $H=G$ of Tucker-Drob's relative
definition of inner amenability \cite[page~5]{Tucker-Drob2020}.

\begin{remark}[]
	The restriction to \emph{atomless} means (instead of just means~$\mu$ 
	that satisfy~$\mu(\{e\}) = 0$, as originally 
	demanded by Effros) has the consequence that fewer groups are inner-amenable (e.g., finite groups are not inner-amenable
	in this sense), but also non-ICC groups \cite[Definition~1.2]{Stalder2006} are not automatically inner-amenable. This is the case, e.g.\ in Stalder's article \cite[below Definition~1.2]{Stalder2006}.
	For ICC groups (i.e., groups where every non-trivial conjugation class is infinite),
	the two notions coincide.
\end{remark}

Similar to amenability, there is a characterisation in terms of F\o lner-sequences.

\begin{lemma}[inner-F\o lner sequence
{{\cite[Théorème~1(F)]{Bedos-Harpe-inner-am}}}]
	\label{lemma:Folner}
	A countable  group $\Gamma$ is inner-amenable
	if and only if it admits an \emph{inner-F\o lner sequence}, i.e., a sequence
	$(F_n)_{n\in \N}$ of finite, nonempty subsets of $\Gamma$ with $\lim_{n\to\infty}
	|F_n| = \infty$
	such that for all $\gamma\in \Gamma$,
	\[
		\lim_{n\to\infty}\frac{|(\gamma\cdot F_n \cdot \gamma^{-1})\triangle F_n|}
		{|F_n|} = 0.
	\]
	
\end{lemma}

Note, that the condition $\mu(\{e\}) = 0$ 
originally translates to the condition~$e \not\in F_n$. In our context, atomlessness of the mean translates
to the property that $|F_n|\to \infty$.

\begin{example}\label{ex:inner-amenable}  
We  collect examples of inner-amenable groups given in the literature.
	
	\begin{enumerate}[(1)]
		\item \emph{Infinite} amenable groups are inner-amenable (because we can choose bi-invariant means). 
		\item Products $A \times \Gamma$, where $A$ is inner-amenable, and $\Gamma$ is any group, are inner-amenable \cite[Corollaire~2(iii)]{Bedos-Harpe-inner-am}.
		\item 
		\label{ex:inner-amenable-ext}		
		Extensions $1 \to \Gamma' \to \Gamma \to \Gamma'' \to 1$, where $\Gamma'$ is inner-amenable and $\Gamma''$ is amenable \cite[Corollaire~2(iv)]{Bedos-Harpe-inner-am}.
		\item Direct limits of inner-amenable groups are inner-amenable \cite[Corollaire~2(v)]{Bedos-Harpe-inner-am}.
		\item 
		\label{ex:inner-amenable-finindex}		
		Let $\Lambda \subseteq \Gamma$ be a finite-index subgroup. Then,
		$\Gamma$ is inner-amenable if and only if $\Lambda$ is \cite[Proposition~2.7]{Duchesne-Cat0}.
		\item All Baumslag-Solitar groups $\operatorname{BS}(m,n)$ (where $m,n\neq 0$) are inner-amenable \cite[Exemple~3.2]{Stalder2006}.
		\item If $H=\Lambda$ is abelian, then every HNN-extension $\operatorname{HNN}(\Lambda, H, K, \phi)$ is inner-amenable \cite[Exemple~3.3]{Stalder2006}.
		\item There is a criterion for the inner amenability of non-ascending HNN extensions 
		\cite[Theorem~1.2]{Duchesne-Cat0}.	Two specific instances of this phenomenon are given
		by Examples of Kida and Ozawa, where the associated subgroups are cyclic 
		\cite[Theorem~1.1 and~1.4]{Kida2014}.
		\item There is a criterion for inner amenability of wreath products \cite[Theorem~1.5]
		{Duchesne-Cat0}.
		\item Groups that are (JS-)stable, McDuff or have property Gamma, are inner-amenable
		 \cite[Figure~1]{Deprez-Vaes2018}. Sufficient conditions for stability can be 
		found in the article by Tucker-Drob \cite[Theorem~18, Corollary~19] {Tucker-Drob2020}.
		\item Thompson's group $F$ is inner-amenable \cite{Jolissaint-ThompsonF}. 
		In fact, it is even stable \cite[Corollary~21]{Tucker-Drob2020}.
		\item Thompson's groups $T$ and $V$ are \emph{not} inner-amenable \cite[Theorem~4.4]{Haagerup2017-ThompsonTV}.
		\item Nonabelian free groups are \emph{not} inner-amenable \cite[Corollaire~3(iii)]{Bedos-Harpe-inner-am}.		
		\item Discrete ICC groups having property (T) are \emph{not} inner-amenable \cite[Corollaire~3(i)]{Bedos-Harpe-inner-am}.
	\end{enumerate}
\end{example}

Another large class of examples was pointed out by Francesco Fournier-Facio.

\begin{example}[]
	\label{ex:fff}
	Let $\Gamma$ be a countable group with commuting conjugates, i.e., for every
	finitely generated subgroup $H\le \Gamma$, there is $f\in \Gamma$ such that
	$H$ commutes with $fHf^{-1}$. 
	Then, $\Gamma$ is inner-amenable. 
\end{example}

\begin{proof}
	We will show that $\Gamma$ admits an inner-F\o lner-sequence (Lemma~\ref{lemma:Folner}).
	Because $\Gamma$ is countable, let $\Gamma~=~\{\gamma_0, \gamma_1, \dots\}$. 
	Since $\Gamma$ has commuting 
	conjugates, for all $n\in \N$, there is $f_n\in \Gamma$ such that
	$\{\gamma_0, \dots, \gamma_n\}$ commutes with 
	$\{f_n\gamma_0f_n^{-1}, \dots, f_n\gamma_nf_n^{-1}\}$. In particular, this 
	implies that $F_n := \{f_n\gamma_0f_n^{-1}, \dots, f_n\gamma_nf_n^{-1}\}$
	defines an inner-F\o lner sequence.
\end{proof}

Together with Example~\ref{ex:inner-amenable}(\ref{ex:inner-amenable-ext}), we obtain
that the following groups are inner-amenable.

\begin{example}[]
	\label{ex:fff2}
	Let $\Gamma$ be a countable group with commuting conjugates, or be a group 
	extension of the form 
	\[
		1 \longrightarrow H \longrightarrow \Gamma \longrightarrow K \longrightarrow 1,
	\]
	where $H$ is countable with commuting conjugates and $K$ is amenable. 
	Then, $\Gamma$ is inner-amenable.
	
	In particular, this includes countable groups of piecewise linear or piecewise projective 
	homomorphisms of the real line \cite[Proof of Theorem~1.3]{FFLodha21} such as Thompson's 
	group~$F$. More examples of this 
	type can be found in the work of Fournier-Facio and Lodha \cite{FFLodha21}.
\end{example}

A natural question to ask is whether the result of Fournier-Facio and Lodha that second bounded cohomology
vanishes for group extensions as in Example~\ref{ex:fff2} \cite[Theorem~1.2]{FFLodha21} extends to all
 inner-amenable groups. Already
the group $\Z \times F_2$ shows that this is \emph{not} the case, as the second bounded cohomology~$H_b^2(\Z \times F_2)$ retracts onto~$H_b^2(F_2) \not\cong 0$.

\section{\texorpdfstring{Transport through $q$-normality}{Transport through q-normality}}
\label{sec:tucker-drob}

In a recent article, Tucker-Drob proved a structure theorem for inner-amenable groups,
 which suggests
a strategy for proving results about inner-amenable groups. The theorem guarantees
the existence of $q$-normal subgroups. Before stating the strategy (Theorem~\ref{thm:Tucker-Drob-method}), we recall the definition of $q$-normality, which was originally introduced
by Popa \cite[Definition~2.3]{Popa2006}.

\begin{defi}[$q$-normal, $q^*$-normal, {{\cite[p.~2]{Tucker-Drob2020}}}]
	\label{def:q-normal}
	A subgroup $H\subset \Gamma$ is \emph{$q$-normal} (resp.\ $q^*$-normal) if there is a generating 
	set $S$ of $\Gamma$, such that for all $s\in S$, the subgroup $sHs^{-1} \cap H$ is infinite (resp.\ 
	\namen).
	
	In this case, we write $H \le_q \Gamma$ (resp.\ $H \le_{q^*} \Gamma$).
\end{defi}

\begin{example}[]
	\label{ex:normal-impl-q-normal}
	Infinite, normal subgroups are $q$-normal. Indeed, if $H \trianglelefteq \Gamma$, then we have
	 $sHs^{-1} \cap H = H$
	for all $s\in \Gamma$, so we can take any generating set as witness. 
	
	The converse is not true. E.g.\ in the 
	\emph{Lamplighter group} $\big(\bigoplus_\Z \Z\big) \rtimes \Z$, the subgroup~$\bigoplus_\N \Z$,
	indexed only over the natural numbers, is $q$-normal but 
	not normal.
\end{example}

We can now state a variant of Tucker-Drob's structure theorem.

\begin{thm}[the structure of inner-amenable groups]
	\label{thm:structure-inner-am}
	Let $\Gamma$ be a finitely generated, virtually torsion-free, inner-amenable and \namen\  group. 
	Then, there exist $g\in \Gamma\backslash \{e\}$ and finitely generated \namen\  subgroups 
	$L\subseteq K \subseteq \Gamma$ such that 
	\[
		\Z \cong \langle g\rangle \le Z(L) \trianglelefteq L \le_q K\le_q \Gamma.
	\]
	Here, $Z(L)$ denotes the centre of $L$.
\end{thm}

\begin{proof}
	We show how to deduce this result from the work of Tucker-Drob \cite{Tucker-Drob2020} 
	(where we always choose $H=\Gamma$ and $\mathcal F = \{\Gamma\}$). 
	The main effort goes into
	proving that a chain of length~$3$ suffices and that $L$ and $K$ can be assumed to 
	be finitely generated. Moreover, we have to show that we can choose $g$ to be
	of infinite order.
	
	Let $\mu$ be an atomless, conjugation-invariant mean on $\Gamma$. Let 
	$\Gamma' \subseteq \Gamma$ be a finite index, torsion-free subgroup. We can
	assume $\mu(\Gamma') = 1$ \cite[Proposition~2.3]{Duchesne-Cat0}.  By a 
	strengthening of a classical theorem, called Rosenblatt's Theorem \cite[Proposition~4.2]{Haagerup2017-ThompsonTV}, the following holds:
	The centraliser $C_\Gamma(g)$ is \namen\ 
	for $\mu$-almost all~$g\in \Gamma$ \cite[Lemma~4.2]{Tucker-Drob2020}. In particular, we 		find such an element $g\in \Gamma'\backslash \{e\}$. Because $\Gamma'$ is torsion-free,
	 we have $\langle g \rangle \cong \Z$.
	Since $C_\Gamma(g)$ is \namen, there is a finitely generated subgroup $L\subset C_\Gamma(g)$ 
	that is \namen. We can suppose that $g\in L$. In particular, $\langle g \rangle$ is central in $L$.
	
	Since $L$ is a \namen\ subgroup of an inner-amenable group, we have that \cite[Theorem~4.3(i)]{Tucker-Drob2020}
	\[
	   L\le_{q^*} \langle L \cup S_L\rangle \le_q \Gamma, \quad
	   \text{where}\quad 
	   S_L \coloneqq \{g\in \Gamma \mid L\cap C_\Gamma (g) \text{ is \namen} \}.
	\]
	Note that in particular, $q^*$-normality implies $q$-normality. We will show how to pick a suitable finitely 
	generated subgroup $K\subset \langle L \cup S_L \rangle$: Since $\langle L \cup S_L\rangle \le_q \Gamma$ 
	and $\Gamma$ is finitely generated, we can pick a finite generating set $g_1, \dots, g_n$ of $\Gamma$
	such that for all $i\in \{1,\dots,n\}$, we have that 
	\[
		g_i\cdot \langle L \cup S_L \rangle \cdot g_i^{-1} \cap \langle L \cup S_L \rangle
	\] 
	is infinite. In particular, its intersection with the finite-index subgroup
	$\Gamma'$ is still infinite and we can pick a non-trivial element in this 
	intersection and write
	\begin{equation}
		\label{eq:conj-in-K-L}
		g_i\cdot w_i \cdot g_i^{-1} = w_i'
	\end{equation}
	where $w_i$ and $w_i'$ are words in $L\cup S_L \cup S_L^{-1}$. 
	Let $\widetilde S_L$ be the (finite) set of letters in~$S_L$ that occur in any of the 
	words $w_1, w_1', \dots, w_n, w_n'$. We set $K \coloneqq \langle L \cup \widetilde S_L\rangle$. 
	As $L$ is 
	finitely generated and $\widetilde S_L$ is finite, $K$ is finitely generated. The relation in
	 Equation~\eqref{eq:conj-in-K-L} witnesses that for all $i\in \{1,\dots,n\}$, the intersection 
	$g_i\cdot K \cdot g_i^{-1} \cap K$ is non-trivial. Because the words $w_i'$ were
	chosen in $\Gamma'$, which is torsion-free, this intersection 
	is infinite. This shows that $K \le_q \Gamma$. Moreover, the generating set $L \cup 
	\widetilde S_L$ 	witnesses that $L \le_q K$. 
\end{proof}

We illustrate this result with an example of a non-trivial chain of $q$-normal subgroups.

\begin{example}[]
	Consider the inner-amenable and \namen\ group \[
		\Gamma \coloneqq F_2 \times F_2 \times \Z 
			= \langle a,b,c,d,e \mid [a,c],[a,d],[b,c],[b,d], [a,e],[b,e],[c,e], [d,e]\rangle.
	\]
	Let $g \coloneqq a$ and $L \coloneqq \langle a,c,d\rangle, 
	K \coloneqq \langle a,b^2,c,d\rangle$.
	Then,
	\[
		\Z \cong \langle g\rangle \le Z(L) \trianglelefteq L \le_q K\le_q \Gamma.
	\]
	Note that $L$ is \emph{not} normal in $K$ and $K$ is \emph{not} normal in $\Gamma$.
\end{example}

Theorem~\ref{thm:structure-inner-am} suggests the following strategy for proving results about virtually torsion-free inner-amenable groups,
which Tucker-Drob used to prove that inner-amenable groups are cheap and of fixed price \cite[Theorem~5]{Tucker-Drob2020}.

\begin{thm}[]
	\label{thm:Tucker-Drob-method}	
	Let $\mathcal C$ be a class of finitely generated, torsion-free groups that 
	is closed under taking finitely generated subgroups and contains the integers~$\Z$.
	Let $P$ be an (isomorphism) invariant defined for groups in $\mathcal C$. 
	Suppose that the following two conditions hold:
	\begin{enumerate}[(1)]
		\item The integers $\Z$ satisfy $P$.
		\label{asspt:TD-1}
		\item Let $L, G \in \mathcal C$ such that $L\le_q G$ (see Definition~\ref{def:q-normal}).
		Then, if $L$ satisfies $P$, then so does $G$.
		\label{asspt:TD-2}
	\end{enumerate}
	
	Then, $P$ holds for all groups in $\mathcal C$ that are inner-amenable but \namen\ .
\end{thm}

\begin{proof}
	Let $\Gamma$ be a virtually torsion-free, finitely generated group in $\mathcal C$ 
	that is inner-amenable 	and \namen.
	By Theorem~\ref{thm:structure-inner-am}, there exist $g\in \Gamma\backslash \{e\}$ and finitely generated subgroups~$L\subseteq K \subseteq \Gamma$ such that
	\[
		\Z \cong \langle g \rangle \le Z(L) \trianglelefteq L \le_q K\le_q \Gamma.
	\]	
	Note that $\langle g \rangle \cong \Z$ satisfies $P$ by assumption~\eqref{asspt:TD-1}. Moreover, $\langle g \rangle$ is central in $L$, thus normal in $L$. In particular, it is also $q$-normal in $L$ (Example~\ref{ex:normal-impl-q-normal}). Thus, we have that  $\langle g \rangle \le_q L \le_q K \le_q \Gamma$. As $L$ and $K$ are finitely 
	generated, $L, K \in \mathcal C$. Hence, we obtain that $\Gamma$ satisfies~$P$ by applying 
	assumption~\eqref{asspt:TD-2} three times. 
\end{proof}

\section{\texorpdfstring{An action induced by $q$-normality}{An action induced by q-normality}}
\label{sec:act-q-normal}
In this section, we will explain how to construct an action out of $q$-normality. If $L \le_q G$ 
is a $q$-normal subgroup, we want to construct an action of~$G$ on a graph whose vertex stabilisers
are (isomorphic to) $L$.

Recall that $G/L \coloneqq \{gL \mid g\in G\}$ is a transitive $G$-set, and an action $G \actson G/L$ 
is given by left-translation of the cosets.

\begin{lemma}[blow up-action out of $q$-normality]
	\label{lemma:act-q-normal-direct}
	Let $G$ be a finitely generated group and 
	let $L\leq_q G$ be a $q$-normal subgroup. There is a
	cocompact action $G \actson (G/L, E)$ 
	on a connected graph whose edge stabilisers are all infinite.
\end{lemma}

\begin{proof}
	We choose a generating set $S$ of $G$ as in Definition~\ref{def:q-normal}.
	Because $G$ is finitely generated, we can assume $S$ to be finite. We define 
	\[
		E \coloneqq \bigl\{\{gL, gsL\} \mid 
			g\in G, s\in S\bigr\}
	\]
	Then, $E$ is closed under $G$-orbits (where the action is by left-translation) 
	and by construction, the action is cocompact (the quotient has one $0$-cell and $|S|$ many
	$1$-cells). Moreover, because $S$ is a generating set of $G$, the graph $(G/L, E)$
	is connected. Finally, fix an edge $\{gL, gsL\}$. Its stabiliser contains
	\[gLg^{-1} \cap (gs)L(gs)^{-1} = g\cdot (L \cap sLs^{-1}) \cdot g^{-1},\]
	which is infinite by choice of $S$. 
\end{proof}

\section{\texorpdfstring{A different approach to actions induced by $q$-normality} {A different approach to actions induced by q-normality}}
\label{sec:more-actions}

In this section, we present an alternative proof of Lemma~\ref{lemma:act-q-normal-direct}.
We show how to get an appropriate action out of a `blow up' constructing, combining
two actions. This exhibits an interesting technique, and shows a more conceptual 
way how to obtain the desired action.
We postpone all technical proofs to Appendix~\ref{appx:proofs}.

If we have a $q$-normal subgroup, we can define the following action on a graph.

\begin{lemma}[action out of $q$-normality]
	\label{lemma:act-q-normal-conj}
	Let $L\leq_q G$ be a $q$-normal subgroup (see Definition~\ref{def:q-normal}).
	Then, we define a graph $\Omega \coloneqq (V,E)$ where
	\begin{align*}
		V &= \{gLg^{-1} \mid g\in G\}\\
		E &= \big\{
			\{g_1Lg_1^{-1}, g_2Lg_2^{-1}\} \mid g_1, g_2 \in G,  g_1Lg_1^{-1}\cap g_2Lg_2^{-1} \text{ is infinite},  g_1Lg_1^{-1}\neq g_2Lg_2^{-1}
			\big\}
	\end{align*}
	Then, $G$ acts on $\Omega$ by conjugation. The vertex set contains exactly one orbit.
	The stabiliser of these vertices is given by the normaliser $N_G(L)$ of $L$ in $G$. 
	
	Moreover, the graph $\Omega$ is connected and for each  $\{g_1Lg_1^{-1}, g_2Lg_2^{-1}\} \in E$,
	we have by construction that $g_1Lg_1^{-1}\cap g_2Lg_2^{-1}$ is infinite.
\end{lemma}

The vertex set consists of exactly one $G$-orbit, and the stabilisers of the vertices
are conjugates of the normaliser $N_G(L)$ of $L$. Thus, as a $G$-set, $V$ is isomorphic 
to $G/N_G(L)$ (where $G$ acts by left-translation). Thus, we will view $\Omega$ as a graph with 
$G$-action and vertex set $G/N_G(L)$ in the following.

On the other hand, the normaliser $N_G(L)$ acts on the classifying space of the quotient $N_G(L)/L$.

\begin{lemma}[action on the quotient]
	\label{lemma:act-quotient}
	Let $L \trianglelefteq H$ be an infinite, normal subgroup. Then, $H$ acts on $E(H/L)^{(1)}$, 
	i.e.\ the 1-skeleton of a classifying space of the quotient $H/L$.
	Explicit constructions show that we can assume the 0-skeleton to be given by $H/L$.
	Moreover, $E(H/L)^{(1)}$ is connected and stabilisers are given by $L$, which is infinite by assumption.
\end{lemma}

We can use these two actions to obtain the statement
 in Lemma~\ref{lemma:act-q-normal-direct} using the following method.

\begin{lemma}[blowing up actions]
	\label{lemma:blow-up}
	
	Let $G$ be a group and $L \subseteq H \subseteq G$ be subgroups.
	Let $H \actson (H/L, E_{H/L})$ and 
	$G \actson (G/H, E_{G/H})$ be actions on connected graphs, where the 
	action on the vertex sets is given by left-translation.
	We suppose the following:
	\begin{itemize}
		\item In $(H/L, E_{H/L})$, all edge stabilisers are infinite. 
		\item In $E_{G/H}$, for each $G$-orbit, we pick a representative $f \in E_{G/H}$ incident 
		to $eH$. We denote by $F$ the set containing the representatives we pick.
		For each $f\in F$, we choose $g(f) \in G$ such that $f = \{eH, g(f)H\}$ and we demand that 
		$L \cap g(f) L g(f)^{-1}$ is infinite.
	\end{itemize}
	
	Then, there is an action $G\actson \Omega \coloneqq (G/L, E)$, given by left-translation 
	on the vertex set $G/L$, where $\Omega$ is a 
	connected graph and each stabiliser of an edge is infinite.
	
	Concretely, we note that $G/L$ is isomorphic (as a $G$-set) to $G \times_H (H/L)$ via the 
	isomorphism
	\[
	\begin{array}{rcl}
		G/L & \longrightarrow & G \times_H H/L\\
		gL & \longmapsto & [(g, eL)] 
	\end{array}
	\]
	and we can thus define $E$ as follows.
	
	\begin{align*}
		E \coloneqq 
			&\big\{
			\{[(g,h_1L)], [(g, h_2L)]\} \mid g\in G, \{h_1L,h_2L\} \in E_{H/L} 
		\big\}\\
		\cup &
		\big\{
			\{[(g, eL)], [(g\cdot g(f), eL)]\} \mid g\in G, f\in F
		\big\}
	\end{align*}
\end{lemma}

To explain notation, we would like to recall the following: 
If $H\actson X$ is an action of~$H$,
	 we can induce an action of~$G$ as follows: We consider the quotient
	\[
		G \times_H X \coloneqq (G\times X) / \big(\forall_{g\in G, x\in X, h\in H} \; (gh, x) \sim (g, hx)\big),
	\]
	which inherits a $G$-action by left-translation on the first component. We call this construction 
	the \emph{induction} of $H\actson X$ by $G$.

We'd like to point out that the resulting action might not be cocompact. 
Because $G$ is finitely generated, we can fix this issue easily.
The idea was pointed out by Gaboriau.

\begin{lemma}[cocompactness]
	\label{lemma:cocompact}
	Let $G$ be finitely generated, $L\subseteq G$ be a subgroup 
	and $G\actson (G/L, E)$ be an action on a connected
	graph induced by left-translation on $G/L$. 
	Then, there is a $G$-invariant subset $E'\subset E$ such that the
	action $G\actson (G/L, E')$ is cocompact and $(G/L, E')$ is connected.
\end{lemma}

It is worth mentioning that in Lemma~\ref{lemma:blow-up}, we could also be more indifferent
about the choice of edges of the seond type, more precisely, we could define the edge set $E$ to be 
\begin{align*}
		E \coloneqq 
			&\big\{
			\{[(g,h_1L)], [(g, h_2L)]\} \mid g\in G, \{h_1L,h_2L\} \in E_{H/L} 
		\big\}\\
		\cup &
		\big\{
			\{[(g_1, eL)], [(g_2, eL)]\} \mid g_1,g_2\in G, \{g_1H, g_2H\} \in E_{G/H}
		\big\}.
\end{align*}

However, with the choice as in Lemma~\ref{lemma:blow-up}, the resulting graph
coincides with the following construction.	

\begin{remark}[]
	The graph $\Omega_{G/H} \coloneqq (G/H, E_{G/H})$ is a $G$-CW complex of 
	dimension~$1$, i.e.\
	a pushout of the following type 		
	\[
	\begin{tikzcd}
		\coprod_{I_1} G/H \times S^0 \ar[r] \ar[d, hook] 			&  G/H \ar[d]  \\
		 \coprod_{I_1} G/H \times D^1 \ar[r]						&  \Omega_{G/H}
	\end{tikzcd}
	\]
	for some index set $I_1$.
	The complex $\Omega$ is then obtained by replacing $G/H$ with the induction of $\Omega_{H/L}
	 \coloneqq (H/L, E_{H/L})$, i.e.,
	$\Omega$ is (the $1$-skeleton of) a pushout as follows
	\[
	\begin{tikzcd}
		\coprod_{I_1} (G\times_H \Omega_{H/L}) \times S^0 \ar[r] \ar[d, hook] 	&  G\times_H \Omega_{H/L} \ar[d]  \\
		 \coprod_{I_1} (G\times_H \Omega_{H/L}) \times D^1 \ar[r]					&  \widetilde\Omega.
	\end{tikzcd}
	\]
	
	In Lemma~\ref{lemma:blow-up}, we just give an explicit description of the resulting edge set.

	This construction was inspired by a similar construction to build classifying spaces: 
	If $N\trianglelefteq 
	G$ is a normal subgroup, and the $G$-CW complex $E(G/N)$ is given, we can blow up by~$G \times_N EN$ 
	(where $EN$ is a model for the classifying space of $N$) to obtain a model of~$EG$. 
	An instance of this method in a slightly different setting is elaborated in an article by Lück 
	and Weiermann \cite[Proof of Proposition~5.1]{Lueck-Weiermann-vcyc}.
\end{remark}

\section{Application to the \coreb}
\label{sec:app-coreb}

In this section, we will prove Theorem~\ref{thm:main:in-am-coreb}. For completeness, we provide
the definition of the \coreb, which is a property of groups, in Appendix~\ref{sec:careb}. However, for the following proof, it suffices
to understand the following two properties of this notion.

\begin{prop}[{{\cite[Lemma~10.10]{ABFG}}}]
	\label{prop:Z-careb}
	The group $\Z$ has the \careb\ for all $\alpha \in \N$.
\end{prop}

\begin{prop}[{{\cite[Example~10.12]{ABFG}}}]
	\label{prop:coreb-act-graph}
	Let $\Gamma$ be a residually finite group that acts cocompactly 
	on a graph $\Omega$ such that 
	\begin{itemize}
		\item vertex stabilisers have the \coreb, and
		\item edge stabilisers are infinite.
	\end{itemize}
	Then, $\Gamma$ has the \coreb.
\end{prop}

We can now prove the following generalisation 
of an inheritance property for normal subgroups \cite[Corollary~10.13 (2)]{ABFG}.

\begin{lemma}[transport through $q$-normality]
	\label{lemma:coreb-inherit-q-normal}
	Let $\Gamma$ be a finitely generated, residually finite group and $L\leq_q \Gamma$ be a $q$-normal subgroup.
	If $L$ has the \coreb, then so does $\Gamma$.
\end{lemma}

\begin{proof}
	It suffices to show the hypotheses of Proposition~\ref{prop:coreb-act-graph}.
	By Lemma~\ref{lemma:act-q-normal-direct}, $\Gamma$ acts cocompactly on a connected graph $\Omega \coloneqq (\Gamma/L, E)$
	such that all edge stabilisers are infinite. All vertex stabilisers are conjugates of $L$, thus
	by assumption have the \coreb. 
\end{proof}

\begin{thm}[]
	\label{thm:in-am-coreb}
	Let $\Gamma$ be a finitely generated, virtually torsion-free, residually finite group
	that is inner-amenable and \namen. Then, $\Gamma$ has the \coreb.
\end{thm}

\begin{proof}
	We apply Theorem~\ref{thm:Tucker-Drob-method} to the 
	class of finitely generated, virtually torsion-free, residually finite groups and the property $P$ 
	`has the \coreb'. It suffices therefore to show that the two hypotheses are
	satisfied.
	
	The group $\Z$ has the \coreb\ by Proposition~\ref{prop:Z-careb}. The second 
	condition is satisfied by Lemma~\ref{lemma:coreb-inherit-q-normal}.
\end{proof}

\begin{remark}[]
	\label{rem:attain-q-coreb}
	Lemma~\ref{lemma:coreb-inherit-q-normal} shows the \coreb\ for a larger class of groups. 
	If $\Gamma$ is a finitely generated, residually finite group and there exist 
	finitely generated subgroups $G_0, \dots, G_n$ such that 
	\[
		\Z \cong G_0 \le_q \dots \le_q G_n = \Gamma,
	\] 
	then $\Gamma$ has the \coreb. 
\end{remark}

\begin{question}[]
	Which groups can be obtained via a sequence of $q$-normal subgroups as in the above remark?
\end{question}

This class contains all virtually torsion-free, inner-amenable and \namen\ groups by Theorem~\ref{thm:structure-inner-am}. This inclusion is strict, as e.g.\ Example~\ref{ex:han}
shows. Moreover, the following groups satisfy 
this condition, recovering a result of \ABFG\ \cite[Proposition~10.15]{ABFG}
where an additional hypothesis of finite generation for some normalisers was necessary.

\begin{cor}
	\label{cor:chain-commuting-coreb}
	Let $\Gamma$ be a residually finite group that is chain-commuting, i.e., there is a finite
	generating set $\{\gamma_1, \dots, \gamma_m\}$ of elements of infinite order such that for 
	all $i~\in~\{1,\dots, m-1\}$, we have $[\gamma_i, \gamma_{i+1}] = e$. Then, $\Gamma$ has the \coreb.
\end{cor}

\begin{proof}
	We have that 
	\[
		\Z \cong \langle \gamma_1 \rangle \le_q \langle \gamma_1, \gamma_2 \rangle 
		\le_q \dots \le_q \langle \gamma_1, \dots, \gamma_m\rangle = \Gamma.
	\]
	Here, $q$-normality is witnessed by the given generating sets and the fact that
	because $\gamma_i$ and $\gamma_{i+1}$ commute, we have
	\[
		\Z \cong \langle \gamma_{i} \rangle \subseteq 
			\gamma_{i+1} \cdot \langle \gamma_1, \dots, \gamma_i\rangle  \cdot \gamma_{i+1}^{-1}
			\cap \langle \gamma_1, \dots, \gamma_i\rangle.
	\]
	Thus, Proposition~\ref{prop:Z-careb} and repeated application of Lemma~\ref{lemma:coreb-inherit-q-normal}
	imply that $\Gamma$ has the \coreb.
\end{proof}

Examples of chain-commuting groups are right-angled Artin groups with connected nerve.
As Damien Gaboriau pointed out to us, in fact all Artin groups with connected nerve
are chain-commuting. They thus have the \coreb, provided that they are
residually finite, recovering a 
special case of a result by \ABFG\ \cite[Theorem~10.17]{ABFG}. 

\begin{cor}
	\label{cor:coreb-of-Artin-groups-conn-nerve}
	Let $\Gamma$ be an Artin group with connected nerve. Then, $\Gamma$ is chain-commuting.
	In particular, if $\Gamma$ is residually finite, it has the \coreb.
\end{cor}

\begin{proof}
	Since the nerve of~$\Gamma$ is connected, we can find a finite sequence of standard generators 
	of $\Gamma$ such that for all two subsequent generators $s,t$, the corresponding vertices
	in the nerve $v_s, v_t$ are connected by an edge. Recall that this means that the 
	Artin group $\langle s, t\rangle_\Gamma$ generated by $s$ and $t$ is spherical. 
	This implies that the centre of $\langle s, t\rangle_\Gamma$ is infinite cyclic \cite[Satz~7.2]{Brieskorn-Saito-AGCG}. In particular, we can choose a generator $\gamma_{s,t}$ in this centre.
	Then, the sequence $(s, \gamma_{s,t}, t)$ is chain-commuting. Choosing a generator of the
	centre for every two subsequent standard generators, we obtain a set of generators of~$\Gamma$
	that is chain-commuting.
	
	We obtain that $\Gamma$ has the \coreb\ by Corollary~\ref{cor:chain-commuting-coreb}.
\end{proof}

\section{Generalisations and Outlook}
\label{sec:outlook}

In Question~\ref{q:higher-degrees-careb}, we asked if the result of Theorem~\ref{thm:main:in-am-coreb}
generalises to higher degrees. We don't expect this to be the case. 
If it does, it would require new insights about inner-amenable groups. By Tucker-Drob's strategy
(Theorem~\ref{thm:Tucker-Drob-method}), a natural approach would be to prove an analogue 
of Lemma~\ref{lemma:coreb-inherit-q-normal} in higher degrees. However, 
already in degree~2, this fails.

\begin{example}[]
	\label{ex:han}
	We have that $\Z \trianglelefteq \Z \times F_2 \le_q F_2 \times F_2$, 
	and $\Z$ as well as $\Z \times F_2$ have the \careb\ for all $\alpha \in \N$ \cite[Lemma~10.10, Corollary~10.13(2)]{ABFG}. The group $F_2 \times F_2$ does \emph{not} have the \ctreb, because 
	its second $\ell^2$-Betti number is positive \cite[Theorem~10.20]{ABFG}. 
	Note that $F_2 \times F_2$ is \emph{not} inner-amenable \cite[Théorème~5]{Bedos-Harpe-inner-am} 
	(alternatively, by Theorem~\ref{thm:raag-inam}).
\end{example}

However, the desired generalisation holds for right-angled Artin groups.
For an introduction to these groups, 
we refer to a survey by Charney \cite{Charney-raags}.
The following is known about inner amenability of right-angled Artin groups.

\begin{thm}[{{\cite[Corollary~4.21]{Duchesne-Cat0}}}]
	\label{thm:raag-inam}
	A right-angled Artin group is inner-amenable if and only if it splits
	as a direct product with $\Z$.
\end{thm} 

We obtain the following corollary.

\begin{cor}[]
	\label{cor:raag}
	Let $\Gamma$ be a right-angled Artin group.
	If $\Gamma$ is inner-amenable, then $\Gamma$ has the \careb\ for all $\alpha\in \N$.
	In particular, also the conclusions of Question~\ref{q:higher-degrees-vanish} hold for $\Gamma$.
\end{cor}

\begin{proof} 
	As $\Gamma$ is inner-amenable, it splits as a direct product with $\Z$
	by Theorem~\ref{thm:raag-inam}. We can then conclude by 
	two lemmas of \ABFG\ 
	\cite[Corollary~10.13(2)]{ABFG} \cite[Lemma~10.10]{ABFG}.
\end{proof}

\appendix

\section{Proofs of Section~\ref{sec:more-actions}}
\label{appx:proofs}

We present the missing proofs of Section~\ref{sec:more-actions}.

\begin{proof}[Proof of Lemma~\ref{lemma:act-q-normal-conj}]
	The only non-obvious claim is connectedness of the graph.
	It suffices to show that for $g\in G$, 
	there is a path from $eLe^{-1}$ to $gLg^{-1}$. Because of $q$-normality 
	(Definition~\ref{def:q-normal}), there is a generating set $S$ of $G$ such that for all $s\in S$, 
	the set $sLs^{-1}\cap L$ is infinite. We pick $s_1, \dots s_n \in S$ and $\epsilon_1, \dots
	\epsilon_n \in \{\pm 1\}$ such that $g = s_1^{\epsilon_1} \dots s_n^{\epsilon_n}$.
	Using that cardinalities are conjugation-invariant, i.e.\ if $sLs^{-1}\cap L$ is infinite, 
	then also for all $g'\in G$, we have that $g'sLs^{-1}g'^{-1} \cap g'Lg'^{-1}$ is infinite, we obtain that
	\[
		eLe^{-1}, s_1^{\epsilon_1}Ls_1^{-\epsilon_1}, \dots, 
		s_1^{\epsilon_1} \cdots s_n^{\epsilon_n} L s_n^{-\epsilon_n} \cdots s_1^{-\epsilon_1} = gLg^{-1}
	\]
	is a path in $\Omega$.	
\end{proof}

\begin{proof}[Proof of Lemma~\ref{lemma:blow-up}]
	By construction, we obtain an action $G \actson \Omega \coloneqq (G\times_H H/L, E)$
	by left-translation on the vertex set. This graph is
	 connected: Let $[(g, eL)] \in G\times_H H/L$. 
	Note that for all $f\in F$, we obtain the edges $\{[(g, eL)], [(g\cdot g(f), eL)]\}$ for 
	all $g\in G$. In particular, since $(G/H, E_{G/H})$ is connected, 
	we obtain a path from $[(g, eL)]$ to $[(h, eL)]$ for some $h\in H$.
	Now, because $(H/L, E_{H/L})$ is connected, we can find a path from 
	$[(h, eL)] = [(e,hL)]$ to $[(e,eL)]$. 
	Concatenating these two paths shows that the graph $\Omega$ is connected.
	 
	It remains to prove that edge stabilisers are infinite: 
	\begin{itemize}
		\item Let $\{[(g,h_1L)], [(g, h_2L)]\}\in E$ with $g\in G, \{h_1L,h_2L\} \in E_{H/L}$. 
		Because edge stabilisers in $(H/L, E_{H/L})$ are infinite, there are infinitely many $h\in H$ that 
		fix $\{h_1L,h_2L\}$.
		For all these $h\in H$, its conjugate $g\cdot h \cdot g^{-1}$ fixes $\{[(g,h_1L)], [(g, h_2L)]\}$. Thus, the stabiliser 
		of this edge is infinite.
		\item Let $g\in G, f\in F$ and consider the
		edge $\{[(g, eL)], [(g\cdot g(f), eL)]\}$.
		This edge is fixed by $g\cdot (L\cap g(f)\cdot L \cdot g(f)^{-1}) \cdot g^{-1}$ which, by the second hypothesis, 
		is an infinite set. \qedhere
	\end{itemize}
\end{proof}

\begin{proof}[Proof of Lemma~\ref{lemma:cocompact}]
	Because $G$ is finitely generated, we can pick a finite generating 
	set $g_1, \dots, g_n$ of $G$. Because $(G/L, E)$ is connected, we can pick paths
	from $eL$ to $g_iL$ for all $i\in \{1,\dots,n\}$. We then define $E'$ to be the union of the 
	$G$-orbits of all edges occuring in one of these paths.
\end{proof}

\section{\texorpdfstring{Cheap $\alpha$-rebuilding property}{Cheap a-rebuilding property}}
\label{sec:careb}

For completeness, we include the definition of the \careb, 
quality of rebuildings and Farber neighbourhoods.
We refer to the work of \ABFG\ \cite{ABFG} for more details and examples.

\begin{defi}[rebuilding {{\cite[Definition~1]{ABFG}}}]
Let $\alpha \in \mathbb N$ and let $Y$ be a CW-complex with 
finite $\alpha$-skeleton. An $\alpha$-\emph{rebuilding} of $Y$ 
is a tuple $(Y,Y',\mathbf{g}, \mathbf{h}, \mathbf{P})$,
consisting of the following data \
\begin{enumerate}
\item $Y'$ is a CW-complex with finite $\alpha$-skeleton, 
\item 
$\mathbf{g} \colon Y^{(\alpha)}\to Y'{}^{(\alpha)} \quad \mbox{and} \quad \mathbf{h} \colon Y'{}^{(\alpha )}\to Y^{(\alpha )}$
are cellular maps
that are homotopy inverse to each other up to dimension $\alpha-1$, i.e., $\mathbf{h} \circ \mathbf{g}_{\restriction Y^{(\alpha-1)}}\simeq \mathrm{id}_{\restriction Y^{(\alpha-1)}}$ within $Y^{(\alpha)}$ and 
$\mathbf{g} \circ \mathbf{h}_{\restriction Y'{}^{(\alpha-1)}}\simeq \mathrm{id}_{\restriction Y'{}^{(\alpha-1)}}$ within $Y'{}^{(\alpha)}$, and
\item a cellular homotopy $\mathbf{P} \colon [0,1] \times Y^{(\alpha-1)}\to Y^{(\alpha)}$  between the identity and $\mathbf{h} \circ \mathbf{g}$,
i.e., $\mathbf{P} (0, \cdot)=\mathrm{id}_{\restriction Y^{(\alpha-1)}}$ and $\mathbf{P} (1,\cdot)=\mathbf{h} \circ \mathbf{g}_{\restriction Y^{(\alpha-1)}}$.
\end{enumerate}
\end{defi}

\begin{defi}[quality of a rebuilding, {{\cite[Definition~2]{ABFG}}}]
Given real numbers~$T, \kappa\geq 1$, we say that an $\alpha$-rebuilding
 $(Y,Y',\mathbf{g}, \mathbf{h}, \mathbf{P})$ is  of 
\emph{quality $({T}, {\kappa})$} if we have for all $j\le \alpha$
\begin{align}
\quad |X'{}^{(j)}|  & \le  {\kappa}{T}^{-1}|X^{(j)}| \tag{cells bound}\\
\quad \log \|g_j \|,\log \| h_j \|,\log \| \rho_{j-1} \|,\log \|\partial'_{j} \|  & \le   {\kappa}(1+\log {T})
\tag{norms bound}
\end{align}
where $|\cdot |$ denotes the number of cells and $\partial'$ 
is the cellular boundary map on $Y'$, and $g$ and 
$h$ are the chain maps respectively associated to $\mathbf{g}$ and 
$\mathbf{h}$, and $\rho \colon C_\bullet(Y)\to C_{\bullet+1}(Y)$ is the chain homotopy induced by $\mathbf{P}$ in the cellular chain complexes: 
\begin{equation*}
\begin{tikzcd}[column sep=1.2em]
C_{\alpha}(Y) \arrow[r, "\partial_{\alpha}"] \arrow[d, "g_{\alpha}" left]
& \cdots \arrow[l, bend left, "\rho_{\alpha-1}" ] & 
\arrow[r] \cdots &  
C_{1}(Y) \arrow[r] \arrow[r, "\partial_1"] \arrow[d, "g_1" left] \arrow[l, bend left, "\rho_1" ] &
C_{0}(Y) \arrow[d, "g_0" left]\arrow[l, bend left, "\rho_0" ]  \\
C_{\alpha}(Y') \arrow[r, "\partial'_{\alpha}"] \arrow[u, xshift=0.35em, "h_{\alpha} " right]
&
\cdots 
& 
\arrow[r] \cdots &  
C_{1}(Y') \arrow[r, "\partial'_1 "] \arrow[u, xshift=0.35em, "h_1" right]&
C_{0}(Y'), \arrow[u, xshift=0.35em, "h_0" right]
\end{tikzcd}
\end{equation*}
and the norms $\| \cdot \|$ are the canonical $\ell^2$-norms on the cellular chain complexes
given by the basis of open cells.
\end{defi}

\begin{defi}[{{\cite[Section~10.1]{ABFG}}}]
	Let $\Gamma$ be a countable group and let
	$\operatorname{Sub}_\Gamma^\text{fi}$ denote
	the space of finite index subgroups of $\Gamma$
	with the topology induced from the topology 
	of pointwise convergence on $\{0,1\}^\Gamma$.
	For $\gamma\in \Gamma$, we consider 
	the following function.
	\[
		\operatorname{fx}_{\Gamma,\gamma}:
		\operatorname{Sub}_\Gamma^\text{fi} \to [0,1],
		\quad 
		\Gamma' \mapsto 
		\frac{|\{g\Gamma'\mid \gamma g \Gamma' = g \Gamma'\}|}{[\Gamma:\Gamma']}
	\]
\end{defi}

We also recall the definition of a Farber sequence.

\begin{defi}[Farber sequence {{\cite[Definition~10.1]{ABFG}}}]
	A sequence $(\Gamma_n)_{n\in \N}$ of subgroups
	of~$\Gamma$ is a \emph{Farber sequence} 
	if it consists of finite index subgroups and 
	for every $\gamma\in \Gamma\backslash \{e\}$,
	we have $\lim_{n\to\infty} \operatorname{fx}_{\Gamma,\gamma} (\Gamma_n) = 0$.
\end{defi}

Note that Farber sequences exist if and only 
if $\Gamma$ is residually finite. The most common 
example of Farber sequences are \emph{residual
chains}, i.e., nested sequences of finite index
normal subgroups whose intersection is trivial.

\begin{defi}[Farber neighbourhood  {{\cite[Definition~10.2]{ABFG}}}]
	Let $\Gamma$ be a residually finite group.
	An open subset $U\subseteq \operatorname{Sub}_\Gamma^\text{fi}$ is a \emph{$\Gamma$-Farber 
	neighbourhood} if it is invariant by the 
	conjugation action of $\Gamma$ on 
	$\operatorname{Sub}_\Gamma^\text{fi}$
	and every Farber sequence in $\operatorname{Sub}_\Gamma^\text{fi}$ eventually belongs to $U$.
\end{defi}

Finally, we can define the \careb.

\begin{defi}[\careb {{\cite[Definition~10.5]{ABFG}}}]
	Let $\Gamma$ be a countable group and $\alpha\in \N$. 
	Then, $\Gamma$ has the \emph{\careb\ }
	if it is residually finite and
	there is a $K(\Gamma,1)$-space $X$ with finite $\alpha$-skeleton 
	and a constant $\kappa_X\ge 1$ such that the following holds:
	For every real number $T\ge 1$, there exists a Farber 
	neighbourhood $U = U(X,T) \subset 
	\operatorname{Sub}_\Gamma^\text{fi}$ such that for every 
	finite covering $Y \to X$ with $\pi_1(Y) \in U$, there
	is an $\alpha$-rebuilding $(Y,Y')$ of quality $(T, \kappa_X)$.  
\end{defi}

{\small
  \bibliographystyle{alpha}
  \bibliography{references}}

\vfill

\noindent
\emph{Matthias Uschold}\\[.5em]
  {\small
  \begin{tabular}{@{\qquad}l}
    Fakult\"at f\"ur Mathematik,
    Universit\"at Regensburg,
    93040 Regensburg, 
    Germany\\
    \textsf{matthias.uschold@mathematik.uni-r.de},
    \textsf{https://homepages.uni-regensburg.de/$\sim$usm34387/}
  \end{tabular}}

\end{document}